\documentclass[12pt,reqno]{amsart}
\usepackage{palatino}
\usepackage{mathpazo}
\usepackage{amsmath,amssymb,mathrsfs}
\usepackage{xcolor}
\colorlet{mdtRed}{red!50!black}
\definecolor{dblue}{rgb}{0,0,.6}
\usepackage[colorlinks,pagebackref=true]{hyperref}
\hypersetup{colorlinks,linkcolor={blue},citecolor={blue},urlcolor={red}}
\renewcommand*{\backref}[1]{}
\renewcommand*{\backrefalt}[4]{[{%
		\ifcase #1 Not cited.%
		\or $\uparrow$~#2.%
		\else $\uparrow$~#2.%
		\fi%
	}]}
\usepackage[all]{xy}
\usepackage{tikz,tikz-cd,tkz-graph,enumerate}

\DeclareMathOperator{\Pic}{\textnormal{Pic}}

\DeclareMathOperator{\Spec}{{\rm Spec}}

\DeclareMathOperator{\mnl}{\mc{M}_C(\textit{n,\,L})}
\DeclareMathOperator{\mnlss}{\mc{M}^{ss}_C(\textit{n,\,L})}
\DeclareMathOperator{\mnls}{\mc{M}^{s}_C(\textit{n,\,L})}

\newcommand{\mf}[1]{\mathfrak{#1}}
\newcommand{\mc}[1]{\mathcal{#1}}
\newcommand{\bb}[1]{\mathbb{#1}}

\newtheorem{theorem}{Theorem}[section]
\newtheorem{lemma}[theorem]{Lemma}
\newtheorem{proposition}[theorem]{Proposition}

\theoremstyle{definition}
\newtheorem{definition}[theorem]{Definition}
\newtheorem{remark}[theorem]{Remark}

\numberwithin{equation}{section}

\usepackage{marvosym} 
\renewcommand{\email}[2][1]{\thanks{\textit{Email address}#1: \href{mailto:#2}{#2}}}

\renewcommand{\address}[2][1]{\thanks{\textit{Address}#1: #2}} 

\newcommand\fnnum[1]{\textsuperscript{#1}}
\makeatother

\begin{document}
	
	\baselineskip=15.5pt
	
	\title[$\bb{A}^1$--connectedness of moduli stack of semi-stable vector bundles]%
	{$\bb{A}^1$--connectedness of moduli stack of semi-stable and parabolic
		semi-stable vector bundles over a curve}
	
	\author[S. Chakraborty]{Sujoy Chakraborty}
	\address[\fnnum{1}]{Department of Mathematics,
		Indian Institute of Science Education and Research Tirupati,
		Andhra Pradesh 517507, India}
	\email[\fnnum{1}]{sujoy.cmi@gmail.com}
	
	\author[S. Holme Choudhury]{Saurav Holme Choudhury}
	\address[\fnnum{2}]{Department of Mathematics,
		Indian Institute of Science Education and Research Tirupati,
		Andhra Pradesh 517507, India}
	\email[\fnnum{2}]{sourav.ac.93@gmail.com}
	
	

	\begin{abstract}
		Let $C$ be an irreducible smooth projective curve of genus $g\geq 2$ over
		an algebraically closed field. We prove that the moduli stack of semi-stable
		vector bundles on $C$ of fixed rank and determinant is $\bb{A}^1$--connected.
		We also show that the moduli stack of quasi-parabolic vector bundles with a
		fixed determinant and a given quasi-parabolic data along a set of points in
		$C$ is $\bb{A}^1$-connected. Moreover, for small and generic weights
		$\boldsymbol{\alpha}$ with $\gcd(n, \deg L) = 1$, the open substack of
		$\boldsymbol{\alpha}$-semistable parabolic vector bundles is also
		$\bb{A}^1$-connected.
	\end{abstract}
	
	\maketitle
	
	\section{Introduction}
	
	Let $C$ be an irreducible smooth projective curve of genus $g\geq 2$ over
	an algebraically closed field $k$. Let $\mc{M}_C(n,L)$ denote the moduli
	stack of vector bundles on $C$ of rank $n$ and determinant $L\in \Pic(C)$.
	By definition, an object in $\mnl$ over a $k$-scheme $S$ is given by a rank
	$n$ vector bundle $\mc{E}$ on $C\times S$, together with an isomorphism
	$\det(\mc{E})\overset{\simeq}{\longrightarrow}p^*L$, where
	$p: C\times S \longrightarrow C$ is the first projection. We have the
	following inclusions of open substacks:
	\[
	\mc{M}^s_C(n,L)\subset \mnlss\subset\mnl,
	\]
	where $\mnlss$ (respectively, $\mc{M}^s_C(n,L)$) denotes the stack of
	semistable (respectively, stable) vector bundles of rank $n$ and determinant
	$L$, whose objects over a $k$-scheme $S$ are given by a vector bundle
	$\mc{E}$ on $C\times S$ whose restriction to each geometric point
	$\textnormal{Spec}\,K\rightarrow S$ is a \textit{semistable} (respectively,
	stable) vector bundle over the curve $C_K$ (see Definition~\ref{def:family}
	for details). It is known that $\mnl$ is a smooth algebraic stack over $k$.
	
	Via the nerve construction, any algebraic stack defines an object in the
	$\bb{A}^1$--homotopy category \cite{MV99} by regarding it as a simplicial
	sheaf. This naturally raises the question of $\bb{A}^1$--connectedness of
	the stacks $\mnl$, as well as for the open substacks $\mnlss$ and
	$\mc{M}_C^s(n,L)$. In \cite{HY24}, the $\bb{A}^1$--connectedness of $\mnl$
	was proved. It was shown that two vector bundles of the same rank are
	$\bb{A}^1$--concordant (see Definition~\ref{def:concordance}) if and only if
	their determinants are isomorphic. This was achieved through looking at
	extensions of vector bundles rather than individual vector bundles, and
	showing that, in a certain sense, any extension of vector bundles can be
	``joined'' to the trivial extension via a line contained in the moduli space
	of extensions between two fixed vector bundles over a projective scheme. The
	key observation here was the fact that the moduli space of such extensions of
	vector bundles is an affine space, which allowed for such a result.
	Unfortunately, this approach does not carry through for the open substack
	$\mnlss$ or $\mc{M}_C^s(n,L)$, as the presence of semistability causes
	additional complications. It is also worth mentioning, as pointed out in
	\cite[Example~2.9]{HY24}, that when $n$ and $\deg(L)$ are coprime, the
	$\bb{A}^1$-connectedness of the stable moduli stack $\mc{M}^{s}_C(n,L)$
	cannot be concluded directly from the rationality of its coarse moduli space
	alone.
	
	There is also the notion of (quasi-)parabolic vector bundles over a curve
	introduced by Mehta and Seshadri \cite{MS80}; these are vector bundles
	together with the data of weighted flags on the fibers over a collection of
	finitely many points on the curve. From the perspective of algebraic
	geometry, the study of the moduli stacks of (quasi-)parabolic vector bundles
	is an important and active area of research.
	
	In this article, we aim to address the question of $\bb{A}^1$--connectedness
	of the following moduli stacks:
	\begin{enumerate}[$\bullet$]
		\item The moduli stack $\mnlss$ of semistable vector bundles over $C$ of
		fixed rank $n$ and determinant $L$.
		\item The moduli stack $\mc{PM}_C^{\boldsymbol{e,m}}(n,L)$ of
		quasi-parabolic vector bundles on $C$ with a fixed determinant and a
		quasi-parabolic data $(\boldsymbol{e,m})$ along a set of parabolic points
		on $C$ (see Section~\ref{section:quasi-parabolic} for details).
		\item The stack $\mc{PM}_C^{\boldsymbol{\alpha}-ss,\boldsymbol{e,m}}(n,L)$
		of $\boldsymbol{\alpha}$-semistable parabolic vector bundles for a system
		of weights $\boldsymbol{\alpha}$; these are open substacks of
		$\mc{PM}_C^{\boldsymbol{e,m}}(n,L)$.
	\end{enumerate}
	The $\bb{A}^1$--connectedness of $\mnlss$ is shown in
	Section~\ref{section:a1-connectedness}. Our proof is based on a result of
	Langton regarding the universal closedness of $\mnlss$ \cite{L75}. To apply
	Langton for the semistable locus, one needs to have in place a direct
	$\bb{A}^1$-concordance between two $K$-points of $\mnlss$ whose semistable
	locus is nonempty. This precludes the usage of the chain in \cite{HY24}, as
	one can show that for $n>2$, the intermediate $\bb{A}^1$s in their
	constructed concordance lie completely in the unstable locus. We avoid this
	by building a direct $\bb{A}^1$-concordance between any two vector bundles
	of the same rank and determinant. If we begin with two semistable vector
	bundles, then Langton applies non-trivially to give a direct
	$\bb{A}^1$-concordance which lies completely in the semistable locus. This
	also provides an alternate proof for a part of the main theorem of
	\cite{HY24}.
	
	In Section~\ref{section:quasi-parabolic}, we prove the
	$\bb{A}^1$--connectedness of $\mc{PM}_C^{\boldsymbol{e,m}}(n,L)$ using the
	concept of flag bundle stacks. In Section~\ref{section:parabolic}, 
	$\bb{A}^1$-connectedness of the moduli stack
	$\mc{PM}_C^{\boldsymbol{\alpha}-ss,\boldsymbol{e,m}}(n,L)$ is shown for small and
	generic weights $\boldsymbol{\alpha}$ with $\gcd(n, \deg L) = 1$. In this
	case, $\boldsymbol{\alpha}$-semistability coincides with
	$\boldsymbol{\alpha}$-stability, the underlying vector bundles are stable,
	and the forgetful morphism to $\mnls$ is an iterated flag bundle stack. The
	$\bb{A}^1$-connectedness then follows from Theorem~\ref{thm:a1-connectedness}
	and Proposition~\ref{prop:flag-bundle-a1-connected}.
	
	\section{$\bb{A}^1$--connectedness of the moduli stack of semistable
		vector bundles over a curve}
	\label{section:a1-connectedness}
	
	\subsection{$\bb{A}^1$-connected components of simplicial sheaves}
	
	\begin{definition}[Sheaf of $\mathbb{A}^1$-connected components]
		\label{def:a1-connected-components}
		Let $S$ be a Noetherian base scheme of finite Krull dimension, and let
		$\mathrm{Sm}_S$ denote the category of smooth schemes of finite type over
		$S$, equipped with the Nisnevich topology. Let $\mathcal{H}_{\bullet}(S)$
		denote the $\mathbb{A}^1$-homotopy category of spaces over $S$, which is the
		homotopy category obtained from the category of simplicial presheaves
		$\mathrm{sPre}(\mathrm{Sm}_S)$ by localizing with respect to
		$\mathbb{A}^1$-weak equivalences, as defined in \cite{MV99}. For a space
		$\mathcal{X} \in \mathcal{H}_{\bullet}(S)$, the \textbf{sheaf of
			$\mathbb{A}^1$-connected components}, denoted $\pi_0^{\mathbb{A}^1}(\mathcal{X})$,
		is the Nisnevich sheaf on $\mathrm{Sm}_S$ defined as the sheafification of
		the presheaf
		\[
		U \longmapsto \mathrm{Hom}_{\mathcal{H}_{\bullet}(S)}(U, \mathcal{X}),
		\]
		where $U \in \mathrm{Sm}_S$. Equivalently, if
		$\mathrm{L}_{\mathbb{A}^1}: \mathrm{sPre}(\mathrm{Sm}_S) \to
		\mathcal{H}_{\bullet}(S)$ denotes the $\mathbb{A}^1$-localization functor,
		then
		\[
		\pi_0^{\mathbb{A}^1}(\mathcal{X}) = a_{\mathrm{Nis}}
		\!\left(U \mapsto \pi_0\!\left(\mathrm{L}_{\mathbb{A}^1}(\mathcal{X})(U)
		\right)\right),
		\]
		where $a_{\mathrm{Nis}}$ denotes Nisnevich sheafification and $\pi_0$
		denotes the set of path components of a simplicial set. A space $\mathcal{X}$
		is called \textbf{$\mathbb{A}^1$-connected} if
		$\pi_0^{\mathbb{A}^1}(\mathcal{X}) \cong *$.
	\end{definition}
	
	\begin{definition}[\text{\cite[Definition~2.6]{HY24}}]
		\label{def:naive-a1-homotopy}
		Let $\mf{X}$ be a simplicial sheaf on $\mathrm{Sm}_k$, the Grothendieck
		site of smooth finite-type schemes over a field $k$ equipped with the
		Nisnevich topology.
		\begin{enumerate}[(1)]
			\item Let $U\in \mathrm{Sm}_k$. Two objects $x$ and $y$ in $\mf{X}(U)$
			are said to be \textit{naively $\bb{A}^1$-homotopic} if there exists a
			map $f: \bb{A}^1\times_k U \rightarrow \mf{X}$ satisfying $f_0 = x$ and
			$f_1 = y$, where $f_i$ is the composition
			$U\xrightarrow{i}\bb{A}^1\times_k U \rightarrow \mf{X}$.
			\item We shall denote by $S'(\mf{X})$ the presheaf on $\mathrm{Sm}_k$
			whose value on a $k$-scheme $U$ is
			\[
			U\mapsto \mf{X}(U)/{\sim},
			\]
			where $\sim$ is the equivalence relation generated by naive
			$\bb{A}^1$-homotopy in~(1). The Nisnevich sheafification of $S'(\mf{X})$
			is denoted by $S(\mf{X})$.
		\end{enumerate}
	\end{definition}
	
	\begin{remark}\label{remark:a1-connected-condition}
		By a result of Morel \cite[Lemma~6.1.3]{M05}, a space $X$ over an infinite
		field $k$ is $\bb{A}^1$-connected if and only if
		\[
		\pi^{\bb{A}^1}_0(X)(\mathrm{Spec}\,F) = *
		\]
		for every finitely generated separable field extension $F$ of $k$. In
		addition, there exists a natural epimorphism $S(X) \to \pi^{\bb{A}^1}_0(X)$.
		Using this, one can show that if $S(X)(\mathrm{Spec}\,F) = *$ for every
		finitely generated separable field extension $F$ of $k$, then $X$ is
		$\bb{A}^1$--connected; this criterion is used in the proof of
		\cite[Theorem~1.1]{HY24}. Moreover, since $S(X)$ is the Nisnevich
		sheafification of $S'(X)$ (see Definition~\ref{def:naive-a1-homotopy}), it
		follows from the properties of the Nisnevich topology that if
		$S'(X)(\mathrm{Spec}\,F) = *$ for every finitely generated separable field
		extension $F$ of $k$, the same holds for its sheafification $S(X)$.
	\end{remark}
	
	\subsection{Valuative criterion of properness for algebraic stacks}
	\label{subsection:valuative-criterion}
	
	Let $\mf{X}$ be an algebraic stack of finite type over $k$. The valuative
	criterion of properness for $\mf{X}$ states the following (see
	\cite[\S~3.1]{ABBLT22}): for every discrete valuation ring $R$ with fraction
	field $K$ and every $2$-commutative diagram
	\begin{align}
		\begin{gathered}
			\xymatrix{
				\textnormal{Spec}\,K \ar[r] \ar[d] & \mf{X} \ar[d] \\
				\textnormal{Spec}\,R \ar[r] & \Spec k,
			}
		\end{gathered}
	\end{align}
	there exists a field extension $K'/K$, a discrete valuation ring
	$R'\subseteq K'$ dominating $R$, and a morphism
	$\Spec R' \longrightarrow \mf{X}$ making the following diagram
	$2$-commutative:
	\begin{align}\label{diagram:valuative-criterion}
		\begin{gathered}
			\xymatrix{
				\Spec K' \ar[r] \ar[d] & \textnormal{Spec}\,K \ar[r] \ar[d]
				& \mf{X} \ar[d] \\
				\Spec R' \ar[r] \ar@{-->}[urr]^(.4){\exists}
				& \textnormal{Spec}\,R \ar[r] & \Spec k.
			}
		\end{gathered}
	\end{align}
	
	\subsection{Semistability of vector bundles}
	
	The notion of semistability and stability of vector bundles over any
	irreducible smooth projective variety over a field is well-known; see
	\cite[\S~2]{L75} for the details. We recall it here for curves.
	
	\begin{definition}\label{def:slope}
		Let $C$ denote an irreducible smooth projective curve over a field $K$, and
		let $E$ be an algebraic vector bundle over $C$.
		\begin{enumerate}[(i)]
			\item The \textit{degree} of $E$, denoted $\deg(E)$, is the degree of the
			divisor class associated to $\det(E)$.
			\item The \textit{slope} of $E$, denoted $\mu(E)$, is the rational number
			$\deg(E)/\mathrm{rank}(E)$.
			\item $E$ is \textit{semistable} (respectively, \textit{stable}) if for
			every proper sub-bundle $F\subsetneq E$ we have
			$\mu(F) \leq \mu(E)$ (respectively, $\mu(F)<\mu(E)$).
		\end{enumerate}
	\end{definition}
	
	\subsection{Concordances}
	
	The notion of $\bb{A}^1$--concordance between two vector bundles was
	introduced in \cite{AKW17}; we recall it below. Note that this notion
	coincides with naive $\bb{A}^1$-homotopy (Definition~\ref{def:naive-a1-homotopy})
	in the context of moduli stacks of vector bundles over a curve.
	
	\begin{definition}[\text{\cite[Definition~5.2]{AKW17}}]\label{def:concordance}
		Let $X$ be a scheme over a field $K$, and choose two distinct points
		$x_0,x_1 \in \bb{A}^1_K$. Two vector bundles $\mc{E}_0$ and $\mc{E}_1$ on
		$X$ are said to be \textit{directly $\bb{A}^1$-concordant} if there exists a vector
		bundle $\mc{E}$ over $X\times \bb{A}^1_K$ such that
		$\mc{E}_0\xrightarrow{\;\simeq\;}\mc{E}|_{X\times\{x_0\}}$ and
		$\mc{E}_1\xrightarrow{\;\simeq\;}\mc{E}|_{X\times\{x_1\}}$. They are
		\textit{$\bb{A}^1$-concordant} if they are equivalent under the equivalence
		relation generated by direct $\bb{A}^1$-concordance.
	\end{definition}
	
	Thus $\mc{E}_0$ and $\mc{E}_1$ are $\bb{A}^1$-concordant if there exists a
	finite sequence $\mc{E}_0=\mc{F}_1, \mc{F}_2, \ldots, \mc{F}_r=\mc{E}_1$
	such that each consecutive pair $(\mc{F}_i,\mc{F}_{i+1})$ is directly
	$\bb{A}^1$-concordant.
	
	\begin{definition}\label{def:ss-concordance}
		Let $X$ be a scheme over a field $K$. Two semistable vector bundles
		$\mc{E}_0$ and $\mc{E}_1$ on $X$ are \textit{directly semistably
			$\bb{A}^1$-concordant} if there exists a family of semistable vector bundles
		$\mc{E}$ over $X\times \bb{A}^1_K$ (see Definition~\ref{def:family} (2))
		such that $\mc{E}_0\xrightarrow{\;\simeq\;} \mc{E}|_{X\times\{x_0\}}$ and
		$\mc{E}_1\xrightarrow{\;\simeq\;} \mc{E}|_{X\times\{x_1\}}$ for two
		distinct points $x_0, x_1 \in \bb{A}^1_K$. They are \textit{semistably
			$\bb{A}^1$-concordant} if they are equivalent under the equivalence relation
		generated by direct semistably $\bb{A}^1$-concordance.
	\end{definition}
	
	The following standard result provides a large supply of direct
	$\bb{A}^1$-concordances; see also \cite[Proposition~2.4]{HY24}.
	
	\begin{proposition}[\text{\cite[Proposition~3.1]{La83}}]
		\label{prop:concordance}
		Let $V=\mathrm{Ext}^1(\mc{E}_1,\mc{E}_0)$, and let
		$p:C\times S \rightarrow C$ denote the projection for a $k$-scheme $S$. The
		functor sending $S$ to the set of isomorphism classes of extensions of
		$p^*\mc{E}_1$ by $p^*\mc{E}_0$ on $C\times S$ is represented by the affine
		space $\bb{V}(V^{\vee})\simeq \bb{A}^N_k$, where $N=\dim_K V$. In
		particular, for any extension $0\rightarrow \mc{E}_0\rightarrow \mc{E}
		\rightarrow \mc{E}_1\rightarrow 0$, the bundle $\mc{E}$ is directly
		$\bb{A}^1$-concordant to $\mc{E}_0\oplus \mc{E}_1$ via the family
		parametrized by the line $t\mapsto t\cdot[\mc{E}]\in \bb{A}^N_K$.
	\end{proposition}
	
	\begin{definition}\label{def:family}
		Let $C$ be an irreducible smooth projective curve over an algebraically
		closed field $k$, and fix a line bundle $L$ on $C$. Let $S$ be a $k$-scheme.
		\begin{enumerate}[(1)]
			\item\label{c1} A \textit{family of vector bundles over $C$ of rank $n$ and
				determinant $L$ parametrized by $S$} is a vector bundle $\mc{E}$ over
			$C\times S$ together with an isomorphism $\det(\mc{E})\simeq p^*L$,
			where $C\times S\xrightarrow{p} C$ is the first projection.
			\item\label{c2} A family $\mc{E}$ as in~\eqref{c1} is a \textit{family of
				semistable vector bundles} if for each geometric point
			$\mathrm{Spec}\,K\rightarrow S$ with $K$ algebraically closed, the
			restriction $\mc{E}|_{C\times\mathrm{Spec}\,K}$ is semistable over
			$C_K := C\times \mathrm{Spec}\,K$.
		\end{enumerate}
	\end{definition}
	
	\begin{remark}
		For a $k$-scheme $S$, the $S$-valued points of $\mnl$ and $\mnlss$ are given
		by families of vector bundles and semistable vector bundles on $C$
		parametrized by $S$ as in Definition~\ref{def:family}, respectively.
	\end{remark}
	
	One of the main inputs for proving the $\bb{A}^1$-connectedness of $\mnlss$
	is the following result, which guarantees the existence of semistable
	quotients. We first record a slope estimate.
	
	\begin{lemma}\label{lem:mu-min}
		There exists a constant $a$, depending only on the curve $C$, such that for
		any two vector bundles $E$ and $F$ over $C$,
		\[
		\mu_{\mathrm{min}}(E\otimes F)\geq \mu_{\mathrm{min}}(E)
		+\mu_{\mathrm{min}}(F) - a\,.
		\]
	\end{lemma}
	
	\begin{proof}
		By \cite[Proposition~2.2]{Ch09}, $\mu_{\mathrm{max}}(E^{\vee}\otimes
		F^{\vee})\leq \mu_{\mathrm{max}}(E^{\vee})+\mu_{\mathrm{max}}(F^{\vee})+a$
		for some constant $a$ depending only on $C$. The claim follows from the
		equality $\mu_{\mathrm{max}}(G^{\vee}) = -\mu_{\mathrm{min}}(G)$ for any
		vector bundle $G$ on $C$.
	\end{proof}

	\begin{proposition}\label{prop:semistable-quotient}
		Let $C$ be a smooth projective curve of genus $g \geq 2$ over an infinite
		field $k$ with a rational point, and let $E$ be a locally free sheaf of rank
		$n \geq 2$ and degree $d$ on $C$. Then for any semistable locally free sheaf $F$ of rank
		$n-1$ with $\det F \cong \det E(nm)$ and $m \gg 0$, there exists a
		surjection $E(m) \twoheadrightarrow F$.
	\end{proposition}
	
	\begin{proof}
		Set $\mathcal{H} := \mathcal{H}om(E(m), F)$. By Lemma~\ref{lem:mu-min}
		there is a constant $a$ depending only on $C$ satisfying
		\begin{align*}
			\mu_{\mathrm{min}}(\mc{H})
			= \mu_{\min}(E^{\vee} \otimes F(-m))
			&\geq \mu_{\min}(E^{\vee}) + \mu_{\min}(F(-m)) - a \\
			&= \mu_{\mathrm{min}}(F(-m)) + c,
			\qquad c:=\mu_{\min}(E^{\vee})-a \\
			&= \mu(F) - m + c \\
			&= \frac{d+mn}{n-1} - m + c \\
			&= \frac{m}{n-1} + \Bigl(\frac{d}{n-1}+c\Bigr)\,,
		\end{align*}
		where $d = \deg(E)$. Hence $\mu_{\min}(\mathcal{H}) \to +\infty$ as
		$m \to \infty$. In particular, by Lemma~\ref{lem:mu-min},
		\[
		\mu_{\max}(\mathcal{H}^\vee(x) \otimes \omega_C)
		= -\mu_{\min}(\mc{H}\otimes \omega_C(x)) < 0
		\]
		for all $x \in C$ and $m\gg 0$. Since a nonzero global section of a bundle
		of negative maximal slope would yield a degree-zero subbundle, we get
		$H^0(C, \mathcal{H}^\vee(x) \otimes \omega_C) = 0$, hence
		$H^1(C, \mathcal{H}(-x)) = 0$ for all $x\in C$ by Serre duality. Thus
		$\mathcal{H}$ is globally generated.
		
		Let $V := H^0(C, \mathcal{H})$; by Riemann--Roch, $\dim V \to \infty$ as
		$m\to \infty$. Consider the incidence locus
		\[
		Z = \{(\phi, x) \in \mathbb{A}(V) \times C \mid \phi_x
		\text{ not surjective}\}.
		\]
		For each $x \in C$, the evaluation map
		$\mathrm{ev}_x : V \twoheadrightarrow \mathrm{Hom}(k^n, k^{n-1})$ is a
		surjective linear map, hence smooth. The locus $Y$ in
		$\mathrm{Hom}(k^n, k^{n-1})$ of non-surjective maps is the determinantal
		variety of maps of rank $\leq n-2$, which has codimension at least $2$. The
		fiber $Z_x$ of the projection $p: Z\longrightarrow C$ is
		$\mathrm{ev}_x^{-1}(Y)$; by smoothness of $\mathrm{ev}_x$, it also has
		codimension at least $2$ in $V$, so $\dim Z_x\leq \dim V - 2$. Since $p$ is
		surjective, by \cite[II~\S~3,~Ex.~3.22]{HartshorneAG}
		\[
		\dim Z \leq \dim Z_x + \dim C = \dim \mathbb{A}(V) - 1 < \dim\mathbb{A}(V).
		\]
		The image of $Z\longrightarrow\mathbb{A}(V)$ is therefore a proper closed
		subvariety. Since $k$ is infinite, there exists $\phi \in V(k)$ in the open
		complement, giving the required surjection $E(m) \twoheadrightarrow F$.
	\end{proof}
	
	\begin{remark}
		The existence of a semistable quotient of this type is likely known to
		experts; see the remark following \cite[Proposition~6.1]{PR03}. We include a
		proof here as our intended application requires the most general assumptions
		on the base field.
	\end{remark}
	
	\begin{proposition}\label{prop:direct-concordance}
		Let $F$ be a field extension of $k$. Let $\mc{E}_0$ and $\mc{E}_1$ be two
		vector bundles of rank $n$ and determinant $p^*L$ over the curve
		$C_F = C\times_k \mathrm{Spec}\,F$, where $p: C_F \rightarrow C$ is the
		projection and $L\in \Pic(C)$. Then $\mc{E}_0$ and $\mc{E}_1$ are directly
		$\bb{A}^1$-concordant.
	\end{proposition}
	
	\begin{proof}
		The curve $C_F$ is an irreducible smooth projective curve over $F$;
		smoothness and projectivity are stable under base change, while
		irreducibility follows from \cite[$\mathrm{IV}_2$,~Proposition~4.5.9]{EGA}.
		Choose any semistable vector bundle $\mc{F}$ of rank $n-1$ and determinant
		$p^*L(mn)$ on $C_F$. By applying Proposition~\ref{prop:semistable-quotient}
		to both $\mc{E}_0$ and $\mc{E}_1$, we may choose $m\gg 0$ so that there
		exist surjections $\mc{E}_i(m) \twoheadrightarrow \mc{F}$, $i=0,1$, with
		kernels $K_i$ satisfying $K_i\simeq \mc{O}_{C_F}$ (by comparison of
		determinants). Twisting by $\mc{O}(-m)$ gives extensions
		\begin{align}\label{eqn:extension}
			0\longrightarrow \mc{O}_{C_F}(-m) \longrightarrow \mc{E}_i
			\longrightarrow \mc{F}(-m) \longrightarrow 0\,,\quad i=0,1.
		\end{align}
		These give rise to two extension classes
		$[\mc{E}_i]\in V := \mathrm{Ext}^1\bigl(\mc{F}(-m),\,\mc{O}_{C_F}(-m)\bigr)$.
		The straight line $t\mapsto (1-t)[\mc{E}_0]+t[\mc{E}_1]$ in $V$ joins the
		two classes, and by Proposition~\ref{prop:concordance} this yields a direct
		$\bb{A}^1$-concordance between $\mc{E}_0$ and $\mc{E}_1$.
	\end{proof}
	
	\begin{remark}
		This strengthens \cite[Theorem~2.5(b)]{HY24}: for vector bundles $\mathcal{E}$
		and $\mathcal{F}$ of rank $n$ over $C$, they are directly $\mathbb{A}^1$-concordant
		if and only if $\det(\mathcal{E}) \simeq \det(\mathcal{F})$. One direction is
		Proposition~\ref{prop:direct-concordance}; for the converse, if $\mathcal{E}$
		and $\mathcal{F}$ are directly concordant, there is an induced morphism
		$\mathbb{A}^1 \to \mathrm{Pic}^d(C)$, which must be constant as $\mathrm{Pic}^d(C)$
		is an abelian variety. Thus $\det(\mathcal{E}) \simeq \det(\mathcal{F})$.
	\end{remark}
	
	\begin{proposition}\label{prop:direct-semistable-concordance}
		With the same setup as in Proposition~\ref{prop:direct-concordance}, assume
		moreover that $\mc{E}_0$ and $\mc{E}_1$ are semistable. Then there exists a
		direct semistable $\bb{A}^1$-concordance between $\mc{E}_0$ and $\mc{E}_1$;
		see Definition~\ref{def:ss-concordance}.
	\end{proposition}
	
	\begin{proof}
		Proposition~\ref{prop:direct-concordance} yields a vector bundle $\mc{E}$ on
		$C_F\times_F\bb{A}^1_F$ and distinct points $x_0, x_1 \in \bb{A}^1_F$
		satisfying Definition~\ref{def:concordance}. Consider the Cartesian diagram
		with $\pi$ and $\pi'$ the first projection maps:
		\begin{align}\label{diagram:diag-1}
			\begin{gathered}
				\xymatrix{
					C\times_{k}\bb{A}^1_F \ar[r]^(.45){\simeq} \ar[d]_{p'}
					& C_F\times_{F} \bb{A}^1_F \ar[rr]^(.6){\pi'} && C_F \ar[d]^{p} \\
					C\times_{k}\bb{A}^1_k \ar[rrr]^(.6){\pi} &&& C.
				}
			\end{gathered}
		\end{align}
		
		From the construction of $\mc{E}$ in Proposition~\ref{prop:direct-concordance},
		it follows that
		\begin{align}\label{eqn:iso-1}
			\det(\mc{E})\ \simeq \ (\pi')^*\det(\mc{E}_i)\ \simeq\ (\pi')^*(p^*L)
			\ =\ (p\circ\pi')^*L = (\pi\circ p')^*L,\quad i=0,1.
		\end{align}
		Via the identification $C\times_{k}\bb{A}^1_F\simeq C_F\times_{F} \bb{A}^1_F$,
		the map $\pi\circ p'$ corresponds to the first projection
		$p_1:C\times_{k}\bb{A}^1_F\longrightarrow C$. Regarding $\mc{E}$ as a
		vector bundle on $C\times_{k}\bb{A}^1_F$, equation~\eqref{eqn:iso-1} gives
		$\det(\mc{E})\simeq p_1^*L$. By Definition~\ref{def:family}, $\mc{E}$ is a
		family of vector bundles on $C$ of rank $n$ and determinant $L$ parametrized
		by $\bb{A}^1_F$, which by the $2$-Yoneda lemma gives a morphism
		\begin{align}\label{eqn:morphism-to-mnl}
			\varphi\ :\ \bb{A}^1_F\ \longrightarrow\ \mnl.
		\end{align}
		By \cite[Proposition~2.11]{ABBLT22}, the semistable locus
		$U := \{p\in \bb{A}^1_F \mid \mc{E}|_{C\times\{p\}} \text{ semistable}\}$
		is an open subscheme of $\bb{A}^1_F$, and $\varphi|_U$ factors through
		$\mnlss$:
		\begin{align}
			\begin{gathered}
				\xymatrix{
					\bb{A}^1_F \ar[rr]^(.45){\varphi} && \mnl \\
					U \ar@{^{(}->}[u] \ar[rr]^(.45){\varphi|_{U}} && \mnlss.
					\ar@{^{(}->}[u]
				}
			\end{gathered}
		\end{align}
		The stack $\mnlss$ is universally closed over $k$ by a result of Langton
		\cite[Proposition~6]{L75} (see also \cite[Theorem~3.7]{ABBLT22}):
		
		\smallskip
		\textit{Let $R$ be a discrete valuation ring with fraction field $K$. For
			any semistable vector bundle $E$ of fixed rank and determinant on $C_K$,
			there exists a vector bundle $\mc{E}$ on $C_R$ with $\mc{E}_K\simeq E$ and
			semistable special fiber.}
		\smallskip
		
		In the language of stacks, this translates via the $2$-Yoneda lemma to:
		\begin{align}
			\begin{gathered}
				\xymatrix{
					\mathrm{Spec}\,K \ar[d] \ar[rr] && \mnlss \ar[d]\\
					\mathrm{Spec}\,R \ar[rr] \ar@{-->}[urr]^{\exists} && \mathrm{Spec}\,k.
				}
			\end{gathered}
		\end{align}
		In other words, the valuative criterion for $\mnlss$ does not require further
		field extensions: it suffices to take $K'=K$ and $R'=R$ in
		diagram~\eqref{diagram:valuative-criterion} (see \cite[\S~3.1]{ABBLT22}).
		Hence $\varphi|_U : U \longrightarrow \mnlss$ extends to a morphism
		\[
		\widetilde{\varphi}\ :\ \bb{A}^1_F \ \longrightarrow\ \mnlss,
		\]
		which by the $2$-Yoneda lemma corresponds to a family $\widetilde{\mc{E}}$
		of semistable vector bundles over $C$ of rank $n$ and determinant $L$
		parametrized by $\bb{A}^1_F$. Since $\mc{E}_0$ and $\mc{E}_1$ are
		semistable, the points $x_0$ and $x_1$ lie in $U$ and are unchanged by
		$\widetilde{\varphi}$. Therefore
		\[
		\widetilde{\mc{E}}|_{C_F\times\{x_j\}}\ \simeq\ \mc{E}_j,\quad j=0,1,
		\]
		showing that $\mc{E}_0$ and $\mc{E}_1$ are directly semistably
		$\bb{A}^1$-concordant.
	\end{proof}
	
	\begin{theorem}\label{thm:a1-connectedness}
		The moduli stack $\mnlss$ of semistable vector bundles of rank $n$ and
		determinant $L$ over an irreducible smooth projective curve $C$ is
		$\bb{A}^1$--connected. In particular, if $n$ and $\deg(L)$ are coprime, the
		moduli stack of stable vector bundles $\mc{M}^s_C(n,L)$ is
		$\bb{A}^1$--connected.
	\end{theorem}
	
	\begin{proof}
		This follows immediately from Remark~\ref{remark:a1-connected-condition} and
		Proposition~\ref{prop:direct-semistable-concordance}.
	\end{proof}
	
	\section{$\bb{A}^1$--connectedness of the moduli stack of quasi-parabolic
		vector bundles over a curve}
	\label{section:quasi-parabolic}
	
	The notion of parabolic vector bundles over a curve was introduced by Mehta
	and Seshadri in \cite{MS80}. As before, $C$ denotes an irreducible smooth
	projective curve over an algebraically closed field $k$. Fix a finite subset
	$D\subset C$ of distinct closed points, called \textit{parabolic points},
	and the following data:
	\begin{enumerate}[$\bullet$]
		\item a positive integer $n$,
		\item for each $p\in D$, a positive integer $e(p)$ and a tuple of positive
		integers $\underline{m}(p):=(m_{p,1},\ldots,\allowbreak  m_{p,e(p)})$
		satisfying $\sum\nolimits_{i=1}^{e(p)}m_{p,i} = n$.
	\end{enumerate}
	Set $\boldsymbol{e} :=\{e(p)\mid p\in D\}$ and
	$\boldsymbol{m} := \{\underline{m}(p)\mid p\in D\}$.
	
	\begin{definition}\label{def:quasi-parabolic-data}
		The tuple $(\boldsymbol{e,m})$ is called a \textit{quasi-parabolic data of
			rank $n$ along $D$}. The collection $\boldsymbol{e}$ is a system of
		\textit{lengths} along $D$, and $\boldsymbol{m}$ is a system of
		\textit{multiplicities} along $D$.
	\end{definition}
	
	\begin{definition}\label{def:quasi-parabolic-bundle}
		Fix a quasi-parabolic data $(\boldsymbol{e,m})$. A \textit{quasi-parabolic
			vector bundle of rank $n$} on $C$ is a rank $n$ vector bundle $E$ on $C$
		together with, for each $p\in D$, a flag of type $\underline{m}(p)$ on the
		fiber $E_p$:
		\[
		E_p = E_{p,1}\supsetneq E_{p,2}\supsetneq \cdots
		\supsetneq E_{p,e(p)} \supsetneq E_{p,e(p)+1} = 0,
		\]
		satisfying $\dim(E_{p,i}/E_{p,i+1})=m_{p,i}$ for all $1\leq i\leq e(p)$.
	\end{definition}
	
	\subsection{Moduli stack of quasi-parabolic vector bundles}
	
	For a vector bundle $\mc{F}$ of rank $n$ on a $k$-scheme $S$ and a tuple
	$\underline{m}=(m_1,\ldots,m_e)$ with $\sum m_i=n$, let
	$\mathrm{Flag}_{\underline{m}}(\mc{F})\longrightarrow S$ denote the
	associated flag bundle of type $\underline{m}$. By its universal property, a
	section of this flag bundle corresponds to a chain of sub-bundles of $\mc{F}$
	of type $\underline{m}$.
	
	\begin{definition}[\text{\cite[Definition~3.1]{BY99}}]
		\label{def:families-of-quasi-parabolic-bundles}
		A \textit{family of quasi-parabolic vector bundles of type $(\boldsymbol{e,m})$
			on $C$ parametrized by $S$} is a vector bundle $\mc{F}$ on $C\times S$
		together with, for each $p\in D$, a section $s_p$ of
		$\mathrm{Flag}_{\underline{m}(p)}\bigl(\mc{F}|_{\{p\}\times S}\bigr)
		\longrightarrow S$.
	\end{definition}
	
	The moduli stack of quasi-parabolic vector bundles on $C$ of rank $n$,
	determinant $L\in\Pic(C)$ and quasi-parabolic data $(\boldsymbol{e,m})$ is
	denoted $\mc{PM}_C^{\boldsymbol{e,m}}(n,L)$. Its objects over $S$ are
	families as in Definition~\ref{def:families-of-quasi-parabolic-bundles} with
	$\det(\mc{F})\simeq p^*L$.
	
	\begin{definition}\label{def:flag-bundle-stack}
		Let $\mf{X}$ be an algebraic stack over $k$, and let $\mc{E}$ be a vector
		bundle of rank $n$ over $\mf{X}$. Fix a tuple
		$\underline{m}=(m_1,\ldots,m_e)$ with $\sum m_i=n$. The \textit{flag bundle
			stack of type $\underline{m}$ associated to $\mc{E}$}, denoted
		$\mathrm{Flag}_{\underline{m}}(\mc{E})$, is the algebraic stack whose
		$T$-points are pairs $(u, F)$ with $u: T\to \mf{X}$ a morphism and $F$ a
		flag of type $\underline{m}$ on $u^*\mc{E}$.
	\end{definition}
	
	\begin{remark}
		If $\mc{E}$ trivializes over a Zariski cover $V\to\mf{X}$, then
		$\mathrm{Flag}_{\underline{m}}(\mc{E})\times_{\mf{X}} V \simeq
		\mathrm{Flag}_{\underline{m}}(k^n)\times V$. Thus
		$\mathrm{Flag}_{\underline{m}}(\mc{E})$ is Zariski-locally trivial with
		fiber the flag variety $\mathrm{Flag}_{\underline{m}}(k^n)$.
	\end{remark}
	
	\begin{proposition}\label{prop:flag-bundle-a1-connected}
		Let $\mf{X}$ be an algebraic stack over an algebraically closed field $k$,
		let $F$ be a finitely generated field extension of $k$, and let
		$f: \mf{Y}\longrightarrow \mf{X}$ be a Zariski-locally trivial flag bundle
		stack with fibers isomorphic to a flag variety $Z$. If
		$S'(\mf{X})(\mathrm{Spec}\,F)=0$, then $S'(\mf{Y})(\mathrm{Spec}\,F)=0$.
	\end{proposition}
	
	\begin{proof}
		Given two $F$-points $x, y \in \mf{Y}$, we show they can be joined by a
		finite sequence of maps $\bb{A}^1_F\longrightarrow\mf{Y}$.
		
		If $f(x)=f(y)$ in $\mf{X}$, then $x$ and $y$ lie in the same fiber of $f$,
		which is a flag variety. Since flag varieties are rational and hence
		$\bb{A}^1$-chain connected, $x$ and $y$ can be joined.
		
		If $f(x)\neq f(y)$, by hypothesis there is a finite sequence of maps
		$\bb{A}^1_F\to\mf{X}$ joining $f(x)$ to $f(y)$; without loss of generality
		assume $f(x)$ and $f(y)$ are directly joined by a morphism
		$g: \bb{A}^1_F\longrightarrow \mf{X}$ with $g(0)=f(x)$ and $g(1)=f(y)$.
		Since $f$ is a representable Zariski-locally trivial flag bundle stack, the
		pullback $g^*\mf{Y}\longrightarrow \bb{A}^1_F$ is the trivial flag bundle
		$\bb{A}^1_F\times_k Z$. Let $x'$ and $y'$ denote the points of
		$\bb{A}^1_F\times_k Z$ corresponding to $(0,x)$ and $(1,y)$ respectively.
		Since $\bb{A}^1_F$ and $Z$ are both $\bb{A}^1$-chain connected, so is their
		product, and $x'$ and $y'$ can be joined by a sequence of maps
		$\bb{A}^1_F\to \bb{A}^1_F\times_k Z \simeq g^*\mf{Y}$. Composing with
		$g^*\mf{Y}\to\mf{Y}$ completes the proof.
	\end{proof}
	\medskip
	
	\begin{theorem}\label{thm:quasi-parabolic-moduli-a1-connected}
		The moduli stack of quasi-parabolic vector bundles $\mc{PM}_C^{\boldsymbol{e,m}}(n,L)$
		is $\bb{A}^1$--connected.
	\end{theorem}
	
	\begin{proof}
		The forgetful map
		\begin{align}\label{eqn:forgetful-map}
			\mc{PM}_C^{\boldsymbol{e,m}}(n,L) \ \longrightarrow\ \mnl
		\end{align}
		is an iterated Zariski-locally trivial flag bundle stack. For $D=\{p\}$ a
		single point, letting $\mc{E}$ be the universal bundle over $\mnl\times C$
		and $\iota_p:\mnl\hookrightarrow\mnl\times C$ the inclusion at $p$, the
		$S$-points of $\mc{PM}_C^{\boldsymbol{e,m}}(n,L)$ and
		$\mathrm{Flag}_{\underline{m}}(\iota_p^*\mc{E})$ coincide for any $k$-scheme
		$S$, giving an isomorphism
		\[
		\mc{PM}_C^{\boldsymbol{e,m}}(n,L)\ \overset{\simeq}{\longrightarrow}\
		\mathrm{Flag}_{\underline{m}}(\iota_p^*\mc{E})\,.
		\]
		For $D=\{p_1,\ldots,p_r\}$, the stack $\mc{PM}_C^{\boldsymbol{e,m}}(n,L)$
		is the fiber product over $\mnl$ of the stacks for each $p_i$, confirming
		that \eqref{eqn:forgetful-map} is an iterated flag bundle stack.
		
		By \cite[Theorem~1.1]{HY24}, $S'(\mnl)(\mathrm{Spec}\,F)=0$ for every
		finitely generated field extension $F/k$. Applying
		Proposition~\ref{prop:flag-bundle-a1-connected} iteratively gives
		$S'(\mc{PM}_C^{\boldsymbol{e,m}}(n,L))(\mathrm{Spec}\,F)=0$, and the
		$\bb{A}^1$-connectedness follows from Remark~\ref{remark:a1-connected-condition}.
	\end{proof}
	
	\section{$\bb{A}^1$--connectedness of the moduli stack of parabolic
		semistable bundles over a curve}
	\label{section:parabolic}
	
	We now consider the open substacks of $\mc{PM}_C^{\boldsymbol{e,m}}(n,L)$
	consisting of $\boldsymbol{\alpha}$-(semi)stable parabolic bundles. Fix a
	quasi-parabolic data $(\boldsymbol{e,m})$ of rank $n$ along $D$, and a
	system of weights
	\begin{align}\label{eqn:weights}
		\boldsymbol{\alpha} :=\{(\alpha_{p,1}<\alpha_{p,2}<\cdots
		<\alpha_{p,e(p)})\mid p\in D\},
	\end{align}
	where $\alpha_{p,i}\in [0,1)$ for all $p$ and $i$.
	
	\begin{definition}\label{def:parabolic-bundles}
		Let $E$ be a quasi-parabolic vector bundle of rank $n$ on $C$ associated to
		$(\boldsymbol{e,m})$ (see Definition~\ref{def:quasi-parabolic-bundle}).
		\begin{enumerate}[(1)]
			\item The $\boldsymbol{\alpha}$-\textit{degree} of $E$ is
			\[
			\deg_{\boldsymbol{\alpha}}(E) := \deg(E)
			+\sum_{p\in D}\sum_{i=1}^{e(p)}m_{p,i}\,\alpha_{p,i}.
			\]
			\item The $\boldsymbol{\alpha}$-\textit{slope} of $E$ is
			$\mu_{\boldsymbol{\alpha}}(E):=\deg_{\boldsymbol{\alpha}}(E)/n$.
		\end{enumerate}
	\end{definition}
	
	The notion of $\boldsymbol{\alpha}$-semistability and
	$\boldsymbol{\alpha}$-stability for a quasi-parabolic vector bundle is
	well-known; see \cite[Definition~1.13]{MS80}. The quasi-parabolic moduli
	stack $\mc{PM}^{\boldsymbol{e,m}}(n,L)$ contains an open substack of
	$\boldsymbol{\alpha}$-semistable bundles, denoted
	$\mc{PM}^{\boldsymbol{\alpha}\text{-}ss,\boldsymbol{e,m}}_C(n,L)$.
	
	\begin{definition}
		A system of weights $\boldsymbol{\alpha}$ is said to be \textit{generic} if a quasi-parabolic vector bundle is $\boldsymbol{\alpha}$-semistable if and only if it is $\boldsymbol{\alpha}$-stable. We shall call a system of weights $\boldsymbol{\alpha}$ to be \textit{small} if it satisfies the numerical condition of \cite[Proposition 5.3]{BY99}. We refer to \cite[\S~2 and \S~5]{BY99} for more details.
	\end{definition}
	 The notion of small and generic weights will be useful for our next result.
	\begin{theorem}\label{thm:parabolic-semistable-moduli-a1-connected}
		For $\boldsymbol{\alpha}$ small and generic and $\gcd(n,\deg L) = 1$, the
		moduli stack $\mc{PM}^{\boldsymbol{\alpha}\text{-}ss,\boldsymbol{e,m}}_C(n,L)$
		is $\bb{A}^1$--connected.
	\end{theorem}
	
	\begin{proof}
		Under these assumptions, $\boldsymbol{\alpha}$-semistability coincides with
		$\boldsymbol{\alpha}$-stability and every semistable bundle of rank $n$ and
		degree $\deg L$ is stable, so
		$\mc{PM}^{\boldsymbol{\alpha}\text{-}ss,\boldsymbol{e,m}}_C(n,L) =
		\mc{PM}^{\boldsymbol{\alpha}\text{-}s,\boldsymbol{e,m}}_C(n,L)$ and
		$\mnlss = \mnls$. By \cite[Proposition 5.3]{BY99}, the forgetful morphism
		\[
		\mc{PM}^{\boldsymbol{\alpha}\text{-}s,\boldsymbol{e,m}}_C(n,L)
		\longrightarrow \mnls
		\]
		is an iterated flag bundle stack. Applying
		Proposition~\ref{prop:flag-bundle-a1-connected} together with
		Theorem~\ref{thm:a1-connectedness} gives
		$S'(\mc{PM}^{\boldsymbol{\alpha}\text{-}s,\boldsymbol{e,m}}_C(n,L))(F) = *$,
		establishing $\bb{A}^1$-connectedness using Remark \ref{remark:a1-connected-condition}.
	\end{proof}
	
	\begin{remark}
		In the above proof we do not need to invoke Langton to repair paths: by
		By \cite[Proposition 5.3]{BY99}, the paths constructed in
		Proposition~\ref{prop:flag-bundle-a1-connected} consist entirely of
		$\boldsymbol{\alpha}$-semistable points.
	\end{remark}
	
	\begin{remark}
		The assumptions of Theorem~\ref{thm:parabolic-semistable-moduli-a1-connected}
		may appear restrictive. For general $\boldsymbol{\alpha}$, the methods of
		this paper do not yield a direct $\bb{A}^1$-concordance between two
		$K$-points of
		$\mc{PM}^{\boldsymbol{\alpha}\text{-}ss,\boldsymbol{e,m}}_C(n,L)$: the
		concordance in Proposition~\ref{prop:flag-bundle-a1-connected} consists of
		multiple $\bb{A}^1$-paths, and it is not clear that each path has a nonempty
		$\boldsymbol{\alpha}$-semistable locus, which would be needed to apply
		Langton.
	\end{remark}
	
	\section*{Acknowledgments}
	
	The first author is supported by the DST INSPIRE Faculty Fellowship (Grant
	No.:~DST/\allowbreak INSPIRE/\allowbreak 04/\allowbreak 2024/\allowbreak 001521),
	Ministry of Science and Technology, Government of India. The second author
	is partially supported by the same grant and would like to thank ISI Kolkata
	for support as Visiting Scientist. The authors thank Rakesh Pawar for
	pointing out an error in a previous version, which led to the substantially
	improved argument presented here.
	

\end{document}